%% file: cyclesystem_reg.tex
\documentclass[12pt]{amsart}

\usepackage{amssymb}
\usepackage{aliascnt}
\usepackage[margin=1in]{geometry}
\usepackage{comment}
\usepackage{todonotes}
\usepackage{tcolorbox}
\usepackage[colorlinks,allcolors=blue]{hyperref}
\usepackage[nameinlink,noabbrev,capitalise]{cleveref}
\usepackage{algorithm}
\usepackage{algpseudocode}
\usepackage{standalone,mathrsfs}

\usepackage{tikz}
\usetikzlibrary{calc}

\usepackage[nameinlink]{cleveref}

\newtheorem{theorem}{Theorem}[section]

\newaliascnt{lemma}{theorem}
\newtheorem{lemma}[lemma]{Lemma}
\aliascntresetthe{lemma}

\newaliascnt{corollary}{theorem}
\newtheorem{corollary}[corollary]{Corollary}
\aliascntresetthe{corollary}

\newaliascnt{proposition}{theorem}
\newtheorem{proposition}[proposition]{Proposition}
\aliascntresetthe{proposition}

\newaliascnt{conjecture}{theorem}
\newtheorem{conjecture}[conjecture]{Conjecture}
\aliascntresetthe{conjecture}

\newaliascnt{question}{theorem}
\newtheorem{question}[question]{Question}
\aliascntresetthe{question}

\theoremstyle{definition}

\newaliascnt{definition}{theorem}
\newtheorem{definition}[definition]{Definition}
\aliascntresetthe{definition}

\newaliascnt{example}{theorem}
\newtheorem{example}[example]{Example}
\aliascntresetthe{example}

\newaliascnt{remark}{theorem}
\newtheorem{remark}[remark]{Remark}
\aliascntresetthe{remark}

\crefname{theorem}{theorem}{theorems}
\Crefname{theorem}{Theorem}{Theorems}

\crefname{lemma}{lemma}{lemmas}
\Crefname{lemma}{Lemma}{Lemmas}

\crefname{corollary}{corollary}{corollaries}
\Crefname{corollary}{Corollary}{Corollaries}

\crefname{proposition}{proposition}{propositions}
\Crefname{proposition}{Proposition}{Propositions}

\crefname{definition}{definition}{definitions}
\Crefname{definition}{Definition}{Definitions}

\crefname{example}{example}{examples}
\Crefname{example}{Example}{Examples}

\crefname{conjecture}{conjecture}{conjectures}
\Crefname{conjecture}{Conjecture}{Conjectures}

\crefname{question}{question}{questions}
\Crefname{question}{Question}{Questions}

\crefname{remark}{remark}{remarks}
\Crefname{remark}{Remark}{Remarks}

\DeclareMathOperator{\cir}{\mathrm{Cir}}
\DeclareMathOperator{\Ind}{Ind}
\newcommand{\CC}{\mathcal{C}}
\newcommand{\ms}[1]{\mathscr{#1}}
\newcommand{\mc}[1]{\mathcal{#1}}
\newcommand{\CCk}[1]{\mathcal{C}^{(#1)}}

\title{Matroids with cycle systems are regular}

 \author{Anton Dochtermann}
 \address{Texas State University} 
 \email{dochtermann@txstate.edu}

 \author{Evan Huang}
 \address{Cornell University} 
 \email{eh677@cornell.edu}

 \author{Kai Mawhinney}
 \address{Harvey Mudd College} 
 \email{kmawhinney3@gmail.com}

 \author{Suho Oh}
 \address{Texas State University} 
 \email{suhooh@txstate.edu}

 \author{Yuchen Xu}
 \address{Amherst College} 
 \email{yuxu30@amherst.edu}

\keywords{Matroids, cycle systems, regular matroids, $h$-vectors, pure $O$-sequences, parking functions}

\usepackage[maxbibnames=99,backend=bibtex]{biblatex}
\begin{document}

\begin{abstract}
A cycle system for a matroid $M$ is a collection of cycles (unions of circuits) whose intersection properties mimic the cut sets of a graph. Cycle systems were introduced by Corry, the first author, McClain, Perkinson, and Yi, who showed that the $h$-vector of any matroid that admits a cycle system is a pure $O$-sequence, confirming a conjecture of Stanley for this class.  Those authors also conjectured that any matroid admitting a cycle system must be binary. Here we answer this conjecture in the affirmative, and prove the stronger result that any such matroid must in fact be regular (representable over any field). From this, we conclude that if $M$ is connected, every cycle system for $M$ is a basis for its circuit space of. We establish other properties of cycle systems along the way, which may be of independent interest.
\end{abstract}

\maketitle

\section{Introduction}

The $h$-vector $h_M = (h_0, h_1, \dots)$ of a matroid $M$ is an important invariant related to the independence complex of $M$, and can also be recovered as an evaluation of its Tutte polynomial. There has been considerable recent progress in understanding the $h$-vectors of matroids. Berget, Spink, and Tseng \cite{BST} showed that the entries form a \emph{log-concave} sequence, generalizing  an earlier result of Huh \cite{Huh}, who established the result for matroids realizable over a field of characteristic zero. An open question of Stanley \cite{Stanley} posits that the $h$-vector of a matroid forms a pure \emph{$O$-sequence}, meaning that there exists an order ideal $\Gamma$ of monomials containing $h_i$ monomials of degree $i$, and such that all maximal elements of $\Gamma$ have the same degree.

As a means to attack Stanley's conjecture, the authors of \cite{CycleSystems} introduced the notion of a \emph{cycle system} of a matroid. The definition depends on the \emph{unique union} $*\{C_1, \dots, C_k\}$ of a collection $C_1, \dots, C_k$ of subsets of the ground set of the matroid, by definition the set of elements in $\bigcup C_i$ that appear in exactly one of the $C_i$ (see \Cref{def:uniqueunion} for details). 
A cycle system on a corank $g$ matroid $M$ is then a collection ${\mathcal C} = \{C_1, \dots, C_g\}$ of cycles (unions of circuits) with the property that the unique union of any nonempty subset of ${\mathcal C}$ is dependent.

In \cite{CycleSystems} the authors show that if ${\mathcal C}$ is a cycle system for a matroid $M$, then there exists a collection ${\mathcal P}^{*}({\mathcal C})$ of \emph{coparking functions}, sequences of nonnegative integers with the property that the number of coparking functions of degree $i$ is given by the $h$-vector entry $h_i$. Furthermore, the \emph{maximal} coparking functions all have the same degree. This establishes Stanley's conjecture for the class of matroids that admit cycle systems, which we will denote by $\ms{C}$.

These constructions extend results of Merino \cite{MerinoTutte}, who established Stanley's conjecture for cographic matroids via the theory of chip-firing and $G$-parking functions. Indeed, if $G$ is a graph with sink vertex $q$ and nonsink vertices $\{1,2,\dots, n\}$, one can show that the collection of cut sets determined by each nonsink vertex $i$ (the set of edges incident to $i$) forms a cycle system for the cographic matroid $M(G)^*$. In addition, the coparking functions in this case recover the $G$-parking functions of the graph $G$. Thus, $\ms{C}$ contains all cographic matroids.

As shown in \cite{CycleSystems}, not all matroids admit cycle systems, including the graphic matroid $M(K_{3,3})$.  In the positive direction, $\ms{C}$ contains all graphic matroids of planar graphs and of \emph{coned} graphs. In addition, cycle systems behave well with respect to various matroid operations including contractions and (certain) deletions, as well as $2$-sums (see below for details). In \cite{CycleSystems} the authors use these facts to show that cycle systems exist for graphic matroids of all $K_{3,3}$-minor-free graphs. 

If $M$ is a connected matroid, it is shown in \cite{CycleSystems} that any cycle system ${\mathcal C}$ on $M$ consists of circuits, and furthermore the elements of ${\mathcal C}$ form a linearly independent collection in the \emph{circuit space} of $M$. In particular, if $M$ is a binary matroid, this implies that ${\mathcal C}$ forms a basis for the circuit space.  
Based on these observations and other calculations, the authors of \cite{CycleSystems} conjecture that any matroid admitting a cycle system must be binary. Our first result is a positive solution to this conjecture. 

\newtheorem*{thm:cycle system binary}{Theorem \ref{thm:cycle system binary}}
\begin{thm:cycle system binary}
All matroids with cycle systems are binary.
\end{thm:cycle system binary}

Our proof utilizes the excluded-minor characterization of binary matroids, as well as new methods for understanding how cycle systems behave with respect to various reduction operations on matroids. After establishing some further properties of cycle systems that follow from \Cref{thm:cycle system binary}, our methods can be pushed further to establish a stronger result.

\newtheorem*{thm:cycle system regular}{Theorem \ref{thm:cycle system regular}}
\begin{thm:cycle system regular}
All matroids with cycle systems are regular.
\end{thm:cycle system regular}

Letting $\ms{B}$ and $\ms{R}$ denote the classes of binary and regular matroids respectively, we can restate \Cref{thm:cycle system regular} as $\ms{C}\subset \ms{R}$ (and recall that $\ms{R}\subset \ms{B}$). We provide the proof of \Cref{thm:cycle system binary} because it is used in the proof of \Cref{thm:cycle system regular}.

In our proofs of \Cref{thm:cycle system binary} and \Cref{thm:cycle system regular} we need a way to rule out the existence of a cycle system for a particular class of matroids. Our main technical tool is \Cref{thm:parallel cancellation}, which allows us to adjoin an additional cycle $D$ to the unique union of some subset of cycle system ${\mathcal C}$ while preserving dependence. Choosing $D$ so that it overlaps the cycles of ${\mathcal C}$ heavily forces a small set to be dependent, and the structure of the underlying matroid then yields the desired contradiction.

Having established that only regular matroids can admit cycle systems, we next investigate how this knowledge can lead to further structural properties. 
For instance, \Cref{thm:cycle system binary} implies that any cycle system ${\mathcal C}$ on a connected matroid $M$ forms a basis for its \emph{circuit space}, which is the ${\mathbb F}_2$-vector space spanned by the indicator vectors of all circuits.  It follows that if $D$ is any circuit of $M$, there is a unique way to express $D$ as a \emph{symmetric difference} of some subset of ${\mathcal C}$.  In \Cref{cor:uniquerep} we prove that this expression is in fact a unique union, meaning that each element of $D$ appears in exactly one element of the subset. We conjecture that no other subset of ${\mathcal C}$ has unique union equal to $D$.

\newtheorem*{conj:unique}{Conjecture \ref{conj:unique}}
\begin{conj:unique}
Suppose $M$ is a matroid with circuit system ${\mathcal C}$, and let $D$ be a circuit of $M$. Then there exists a unique $S\subseteq [g]$ such that $D = *\{C_i: i \in S\}$. 
\end{conj:unique}

In the final part of the paper, we make progress towards \Cref{conj:unique}. In particular, we provide the following sufficient condition, which may be more tractable. In the following, we use $\mc C_{\ge 3}$ to denote the elements of $E(M)$ that appear in at least 3 distinct cycles of $\mc C$.

\newtheorem*{conj:C>=3 ind}{Conjecture \ref{conj:C>=3 ind}}
\begin{conj:C>=3 ind}
    Let $M$ be a matroid with circuit system $\mathcal{C}$, where $\mathcal{C}_{[g]}$ is a circuit. Then $\mathcal{C}_{\ge 3}$ is independent.
\end{conj:C>=3 ind}

We prove that the second conjecture implies the first.

\newtheorem*{thm:sufficient unique}{Theorem \ref{thm:sufficient unique}}
\begin{thm:sufficient unique}
    \Cref{conj:C>=3 ind} implies \Cref{conj:unique}.
\end{thm:sufficient unique}

\subsection{Organization}
The rest of the paper is organized as follows. In \Cref{sec:prelim} we provide background material, including a review of basic matroid theory, $h$-vectors, and classes of matroids needed for our study. Here we also review the notions of unique unions and cycle systems for matroids. In \Cref{sec:tools} we collect some technical tools that we will need in the proofs of our main theorems.  In \Cref{sec:binary} we focus on the binary case and use the minor characterization and our reduction tools to prove \Cref{thm:cycle system binary}, as well as derive some corollaries. In \Cref{sec:regular} we prove our main result \Cref{thm:cycle system regular}. In \Cref{sec:uu and symdiff} we address the uniqueness question of representing circuits in terms of elements of a cycle system, and prove \Cref{prop:C>=3}. Finally, in \Cref{sec:Further} we discuss some open questions and possible avenues for future research.

\section{Preliminaries}\label{sec:prelim}

\subsection{Basic matroid theory}

We begin with some background material regarding matroids. Although we provide most definitions, we will assume that the reader is familiar with the basics of matroid theory. We mostly follow the treatment in \cite{Oxley}.

\begin{definition}[Matroid]
A \emph{matroid} $M = (E,{\mathcal I})$ on a finite \emph{ground set} $E=E(M)$ is a collection $\mathcal{I}=\mathcal{I}(M)$ of subsets of $E$ satisfying the following properties:
\begin{enumerate}
    \item $\emptyset\in\mathcal{I}$;
    \item If $X \in {\mathcal I}$ and $Y \subseteq X$ then $Y \in {\mathcal I}$;
    \item If $X, Y \in {\mathcal I}$ and $|X| > |Y|$ then there exists $e \in X \setminus Y$ such that $Y \cup \{e\} \in {\mathcal I}$. (Exchange property)
\end{enumerate}
\end{definition}

The elements of ${\mathcal I}$ are called the \emph{independent sets} of $M$, denoted $\mathcal{I}(M)$.
Two matroids are \emph{isomorphic} if there is a bijection of their ground sets
inducing a bijection of independent sets.

An independent set that is maximal (under inclusion) is called a \emph{basis}. The number of elements in any (and hence every) basis of $M$ is called the \emph{rank} of the matroid, denoted $r(M)$. More generally, for $S \subseteq E$, the \emph{rank} of $S$, denoted $r(S)$, is the maximum cardinality of an independent subset of $S$. The \emph{corank} of a matroid $M$ is given by $g = |E| - r(M)$. A subset $F \subseteq E$ is a \emph{flat} if adding any element of $E \setminus F$ increases rank, so that $r(F \cup \{x\}) = r(F) + 1$ for any $x \in E \setminus F$. Finally, a \emph{hyperplane} is defined as a maximal proper flat, and a \emph{cocircuit} is defined as the complement of a hyperplane.

A subset of $E$ is \emph{dependent} if it is not independent, and a minimal dependent subset is called a \emph{circuit}. Note that a set is dependent if and only if it contains a circuit.  The collection of circuits of $M$, denoted $\cir(M)$, determines the independent sets of $M$, and in fact a matroid can equivalently be defined by the following properties:
\begin{enumerate}
    \item $\emptyset\notin\cir(M)$;
    \item If $C,C'\in\cir(M)$ and $C\subseteq C'$, then $C=C'$;
    \item\label{item:circuit3} (Circuit elimination property) If $C,C'\in\cir(M)$ with $C\neq C'$ and $e\in C\cap C'$, then there exists $D\in\cir(M)$ such that $D\subseteq (C\cup C')\setminus \{e\}$.
\end{enumerate}

A \emph{loop} of a matroid $M$ is an element $e\in E$ that is contained in no basis. A \emph{coloop} (or sometimes \emph{bridge}) is an element $e\in E$ contained in every basis. An element is a coloop precisely when it is not contained in any circuit. A \emph{cycle} is a subset of \(E\) that can be expressed as the union of circuits. 

For a matroid $M = (E, {\mathcal I})$, the \emph{dual matroid} $M^*$ has ground set $E$ and bases given by complements of the bases of $M$, so that ${\mathcal B}(M^*) = \{E \setminus B:B \in {\mathcal B}(M)\}$. With this language one can check that a cocircuit of $M$ is precisely a circuit in the dual matroid $M^*$. The following property relating circuits and cocircuits of a matroid will be useful for our study.

\begin{lemma}[{\cite[Proposition 2.1.11]{Oxley}}]\label{lem:intersection}
Suppose $M$ is a matroid, and let $C$ be a circuit of $M$ and $D$ a cocircuit of $M$. Then $|C \cap D| \neq 1$.
\end{lemma}
Given an undirected graph $G$ on vertex set $V$ and edge set $E$ (allowing loops and parallel edges), one can form the \emph{graphic matroid} $M(G)$ whose ground set is $E$ and whose independent sets are given by the collections of edges that form a forest in $G$. More generally, a matroid $M$ is \emph{graphic} if it is isomorphic to the graphic matroid of some graph. The dual matroids of graphic matroids are called \emph{cographic} matroids. 

Many notions from graph theory have matroidal analogues. For instance, the operation of deleting or contracting an edge in a graph to obtain a new graph leads to the definition of matroid minors.

\begin{definition}[Minors]
Suppose $M$ is a matroid on ground set $E$ and let $S\subseteq E$. The \emph{deletion} $M\setminus S$ is the matroid on ground set $E \setminus S$, with independent sets $\{I \subseteq E \setminus S: I \in {\mathcal I}\}$.  The \emph{contraction} is defined as $M/S = (M^*\setminus S)^*$, again on ground set $E\setminus S$. A matroid $N$ is a \emph{minor} of the matroid $M$ if it can be constructed from $M$ by a sequence of deletions and contractions.
\end{definition}

The notion of parallel edges of a graph also has a matroidal interpretation.

\begin{definition}[Parallel elements]
    Suppose $M$ is a matroid. Elements $a,b\in E(M)$ are said to be \emph{parallel} if $\{a,b\}\in \cir(M)$. The relation $a\sim b\iff\{a,b\}\in \cir(M)$ or $a=b$ is an equivalence relation, and the equivalence classes are called \emph{parallel classes}. 
\end{definition}

It can be verified that if $a$ and $b$ are parallel, and if $I$ is an independent set containing $a$, then $(I \setminus \{a\}) \cup \{b\}$ is also independent. Similarly, if $C$ is dependent with $a \in C$ and $b\notin C$, then $(C \setminus \{a\}) \cup \{b\}$ is also dependent.

Given a graph with multiple loops and parallel edges, many of the underlying matroidal properties can be read off from the simple graph obtained by deleting the loops and all but one element from each parallel class. This idea is generalized to matroids using the following notion of simplification.

\begin{definition}[Simplification]\label{def:simplification}
    Given a matroid $M$, the \emph{simplification} $\widetilde{M}$ of $M$ is obtained by deleting all loops and all but one representative from each parallel class. 
\end{definition} 

\begin{remark}\label{rem:simplification}
    The simplification of $M$ is well-defined up to isomorphism regardless of which specific representatives are chosen. 
\end{remark}    

\subsection{$h$-vectors of matroids}

The $h$-vector of a matroid can be defined in a number of ways.  For one approach, note that the collection of independent sets of a rank $d$ matroid $M$ forms a $(d-1)$-dimensional simplicial complex $\Ind(M)$ called the \emph{independence complex} of $M$.  The \emph{f-vector} $\vec{f} = (f_{-1}, f_0, \dots, f_{d-1})$ of $\Ind(M)$ counts the number of simplices of $\Ind(M)$ of each dimension, so that $f_{i-1}$ is the number of independent sets of $M$ of size $i$. The \emph{$h$-vector of $M$} is then by definition the $h$-vector of this simplicial complex, which is often a more convenient way to encode the data of $\vec{f}$.  In particular, one can define the entries of $\vec{h} = (h_0, \dots, h_d)$ according to the linear relation
\[\sum_{i=0}^d f_{i-1}(t-1)^{d-i} = \sum_{k=0}^d h_kt^{d-k}.\]
\noindent

One can also recover the $h$-vector of $M$ as an evaluation of its Tutte polynomial. For this, recall that if $M$ is a matroid on ground set $E$, the \emph{Tutte polynomial} of $M$ is the bivariate polynomial 
\[T_M(x,y) = \sum_{S \subseteq E} (x-1)^{r(M)-r(S)}(y-1)^{|S|-r(S)}.\]
If we evaluate this expression at the value $y =1$, we get
\[T_M(x,1) = \sum_{I \in {\mathcal I}} (x-1)^{r(M)-|I|} = \sum_{i=0}^d f_{i-1}(x-1)^{d-i} = \sum_{k=0}^d h_kx^{d-k}.\]
In particular we see that $T_M(x,1)$ is a polynomial in $x$ whose coefficients are the entries of the $h$-vector (in reverse order).
Also, if $M^*$ is the matroid dual to $M$, we have $T_{M^*}(x,y) = T_M(y,x)$. From this we see that $T_M(1,y)$ is a polynomial that recovers the $h$-vector of the dual $M^*$.

The $h$-vector of a matroid $M$ can also be defined in terms of \emph{activity} of the bases of $M$. In particular one can show that $h_i$ is the number of bases of $M$ that have $i$ \emph{internally passive} elements, under any ordering of the ground set. We omit the details here, since we will not need this approach in our work.

\subsection{Classes of matroids and characterizations}
\label{sec:classes}
We next review some classes of matroids that will be important for our study.  A \emph{representation} of a matroid $M$ is a family ${\mathcal F}$  of vectors  in a vector space $V$ with the property that the independent sets of $M$ correspond to the subsets of ${\mathcal F}$ that are linearly independent. A matroid is \emph{linear} if it has a representation, and $M$ is said to be \emph{$F$-linear} (for some field $F$) if it has a representation in a vector space over $F$. 

A matroid $M$ is \emph{binary} if it can be represented over the field of two elements, which we denote ${\mathbb F}_2$. A matroid is \emph{regular} if it can be represented over all fields. One can show that both graphic and cographic matroids are regular. There are many different characterizations of binary and regular matroids; for us, the following will be most useful. Recall that the \emph{uniform matroid $U_{k,n}$} is by definition the matroid on ground set $E = [n]$ with bases given by all $k$-subsets. 

\begin{proposition}[{\cite{Tutte1}}]\label{prop:binary u24}
    A matroid is binary if and only if it does not contain the uniform matroid $U_{2,4}$ as a minor.
\end{proposition}

Binary matroids enjoy a stronger version of circuit elimination, sometimes referred to as \emph{double circuit elimination}.

\begin{proposition}[{\cite[Corollary 9.1.6]{Oxley}}]\label{prop:double elimination}
    A matroid \(M\) is binary if and only if whenever \(C_1\) and \(C_2\) are distinct circuits of \(M\) with $\{e,f\} \subseteq (C_1\cap C_2)$, there is a circuit contained in \((C_1\cup C_2) - \{e,f\}\).
\end{proposition}

Binary matroids also have the following characterization that involves intersecting circuits and cocircuits.

\begin{proposition}[{\cite[Theorem 9.1.2]{Oxley}}]\label{prop:cocircuit}
    A matroid \(M\) is binary if and only if the size of the intersection of any circuit and cocircuit of \(M\) is even.
\end{proposition}

The class of regular matroids also has an excluded-minor characterization.

\begin{proposition}[{\cite{Tutte1}}]\label{prop:regular minors}
A matroid $M$ is regular if and only if it does not contain $U_{2,4}$, $F_7$, or $F_7^*$ as a minor.
\end{proposition}

Here $F_7$ denotes the \emph{Fano matroid}, a rank $3$ matroid whose size-3 dependent sets are given by the seven 3-sets of points corresponding to lines in the Fano plane, depicted in \Cref{fig: Fano and Kyle}. We let $F_7^*$ denote its dual. 

\begin{figure}[H]
    \centering

\begin{tikzpicture}[scale=1.4, line cap=round, line join=round]
  \tikzset{pt/.style={circle, fill=black, inner sep=1.8pt}}

  \coordinate (O) at (0,0);
  \coordinate (A) at (90:2);
  \coordinate (B) at (210:2);
  \coordinate (C) at (330:2);
  \coordinate (AB) at ($(A)!0.5!(B)$);
  \coordinate (BC) at ($(B)!0.5!(C)$);
  \coordinate (CA) at ($(C)!0.5!(A)$);

  \draw[thick] (A)--(B)--(C)--cycle;
  \draw[thick] (A)--(BC);
  \draw[thick] (B)--(CA);
  \draw[thick] (C)--(AB);
  \draw[thick] (O) circle[radius=1];
  
  \node[label={[xshift=0.2cm]4}] (O) at (O) {};
  \node[label={[yshift=-0.2cm]3}] (A) at (A) {};
  \node[label={1}] (B) at (B) {};
  \node[label={7}] (C) at (C) {};
  \node[label={[xshift=-0.2cm,yshift=-0.2cm]2}] (AB) at (AB) {};
  \node[label={[xshift=0.2cm]6}] (BC) at (BC) {};
  \node[label={[xshift=0.2cm,yshift=-0.2cm]5}] (CA) at (CA) {};
\end{tikzpicture}
\qquad \qquad
\begin{tikzpicture}[scale=1.5, line cap=round]
  \tikzset{
    vertex/.style={circle, fill=black, inner sep=1.6pt},
    elabel/.style={fill=white, inner sep=1.8pt}
  }
  
  \node[vertex] (A) at (0,0) {};
  \node[vertex] (B) at (0,2) {};
  \node[vertex] (C) at (2,2) {};
  \node[vertex] (D) at (2,0) {};

  \draw[thick] (A) -- node[elabel, below] {1} (D)
        -- node[elabel, right] {2} (C)
        -- node[elabel, above] {3} (B)
        -- node[elabel, left]  {4} (A);

  \draw[thick] (B) -- node[elabel, above] {5} (D);
\end{tikzpicture}
\caption{The Fano plane with labeled \textbf{vertices}; and the graph $G = K_4 \setminus e$ with labeled \textbf{edges}.}\label{fig: Fano and Kyle}
    \end{figure}
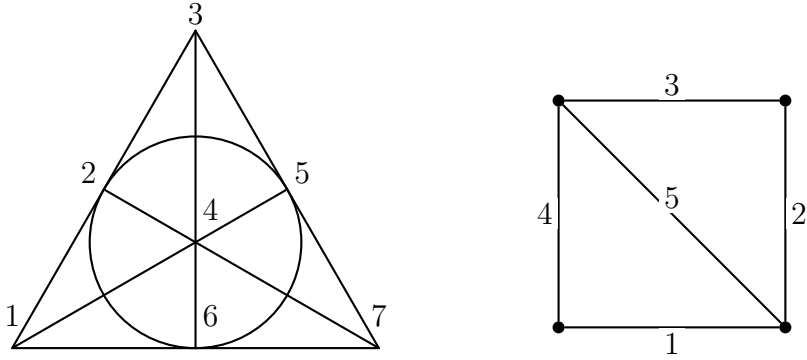

One can show that both $F_7$ and $F_7^*$ are representable over all fields of characteristic $2$, but over no field of characteristic different from $2$.

We will often omit brackets and commas in our description of the subsets of our ground sets \(E\), especially if those elements are integers from 1 to 9. For example, we write $247$ to denote $\{2,4,7\}$. We will use the following explicit binary representations for the matroids $F_7$ and \(F_7^*\), where the columns of each matrix provide the collections of vectors.
    \[
\left[
\begin{array}{ccccccc}
\scriptstyle 1&\scriptstyle 2&\scriptstyle 3&\scriptstyle 4&
\scriptstyle 5&\scriptstyle 6&\scriptstyle 7\\
\hline
0&1&1&1&1&0&0\\
1&1&0&1&0&1&0\\
1&0&1&1&0&0&1
\end{array}
\right]\qquad\left[
\begin{array}{ccccccccc}
\scriptstyle 1&\scriptstyle 2&\scriptstyle 3&\scriptstyle 4&
\scriptstyle 5&\scriptstyle 6&\scriptstyle 7\\
\hline
1&0&0&0&0&1&1 \\
0&1&0&0&1&1&0 \\
0&0&1&0&1&0&1 \\
0&0&0&1&1&1&1
\end{array}
\right].
\]
With this labeling, the circuits of $F_7$ are
\[123,145,167,247,256,346,357,\]
\[1246,1257,1347,1356,2345,2367,4567\]
and the circuits of $F_7^*$ are
\[1246,1257,1347,1356,2345,2367,4567.\]
The Fano plane diagram in \Cref{fig: Fano and Kyle} is labeled in accordance with this set of circuits.

\begin{remark}\label{rem:Fano}
Our arguments will also exploit the symmetries of the Fano plane. For example, suppose we are fixing a circuit $C$ in $F_7$. In this case we can assume without loss of generality that $C = 123$ if the circuit has size 3, or $C = 1246$ if it has size 4. Similarly, in the dual $F_7^*$, we may assume without loss of generality that a chosen circuit is $1246$. 
\end{remark}

In addition, if $M$ is a matroid on ground set $E = [n]$ containing either $F_7$ or $F_7^*$ as a minor, we will assume without loss of generality that the first seven elements of the ground set $E(M)$ are the elements of that minor and use the binary representations from above for those elements. 

We next recall the definition of the matroid $AG(3,2)$.

\begin{definition}\label{def:A32}
    The matroid \(AG(3,2)\) is defined on ground set \(\mathbb{F}_2^3\), with circuits given by the affine planes of $\mathbb{F}_2^3$. In this paper, we will label the ground set \([8]\) and use the circuits
    \[1246,1257,1347,1356,2345,2367,4567\]
    \[1238,1458,1678,2478,2568,3468,3578.\]
\end{definition}

Importantly, all elements of $AG(3,2)$ are symmetric. Also, for any element $e$ one can check that $AG(3,2) \setminus e\cong F_7^*$, whereas $AG(3,2) / e \cong F_7$.

\subsection{Unique unions and cycle systems}

We next recall the notion of a cycle system on a matroid and discuss some relevant properties. Cycle systems were introduced by Corry, the first author, McClain, Perkinson, and Yi in \cite{CycleSystems}, where they were used to study the $h$-vectors of matroids. A motivating example of a cycle system comes from any basis for the cut space of a connected $n$-vertex graph $G$ given by the incident sets (coboundaries) of a collection of $n-1$ vertices. In this case the elements of the basis have certain properties that can be phrased in terms of a `unique union operator'.

\begin{definition}[Unique Union Operator]\label{def:uniqueunion}
Suppose $A_1,\dots,A_k$ are subsets of a ground set $E$. For $S \subseteq [k]$, we define the \emph{unique union}, denoted
    \[\mathcal{A}_S\quad\text{or}\quad \ast\{A_i : i\in S\}\]
    to be the subset of $E$ consisting of all elements that appear in exactly one of the $A_i$ for $i \in S$. 
\end{definition}

For example, if $A_1 = \{1,2,3\}, A_2=\{3,4\}, A_3 = \{2,3,4,5,6\}$, then we find:
    \[\mathcal{A}_{\{1,2\}} = \{1,2,4\}, \; \; \mathcal{A}_{[3]} = \{1,5,6\}.\]
    
\begin{definition}[Cycle System]
    Let $M$ be a matroid with corank $g$. A cycle system $\mathcal{C}$ on $M$ is a set of $g$ cycles $C_1,\dots,C_g$ such that for any nonempty subset $S\subseteq [g]$,
    \[\mathcal{C}_S \not\in \mathcal{I}(M).\]
\end{definition}

For example, if $M = M(G)^*$ is a cographic matroid with underlying connected graph $G$, one can fix a sink vertex $q$ and for each nonsink vertex $v_i$ let $C_i$ denote the set of edges incident to $i$ (the cut set determined by $v_i$). One can then check that the set $\mathcal{C} = \{C_1, \dots, C_g\}$ is a cycle system for $M$. In fact this `overlap property' of cut sets is what inspired the definition of cycle systems.

One can also see that if $G$ is a connected planar graph, the set of bounded faces $\{F_1, \dots, F_g\}$ in an embedding of $G$ defines a cycle system for the graphic matroid $M(G)$. In \cite{CycleSystems} it is shown that the graphic matroid of any `cone graph' also admits a cycle system.  Not all matroids admit cycle systems (for example the graphic matroid of $K_{3,3}$), and some matroids may admit multiple cycle systems. 

\begin{remark}
    We use the following notation for various classes of matroids:
    \begin{itemize}
        \item $\ms{C}$ is the class of matroids with cycle systems;
        \item $\ms{R}$ is the class of regular matroids;
        \item $\ms{B}$ is the class of binary matroids;
        \item $\ms{G}$ is the class of graphic matroids;
        \item $\ms{G^*}$ is the class of cographic matroids.
    \end{itemize}
\end{remark}

The main goal of this paper is to understand the class $\ms{C}$. We say ${\mathcal C}$ is a \emph{circuit system} if all its elements are circuits.

\begin{example}
Let $M = M(G)$ be the graphic matroid of the graph $G = K_4 \setminus e$ depicted in \Cref{fig: Fano and Kyle}. If we let $C_1 = \{1,4,5\}$ and $C_2 = \{2,3,5\}$, one can verify that the collection $\mathcal{C} = \{C_1, C_2\}$ forms a cycle (in fact, circuit) system for $M$.  In particular we have that ${\mathcal C}_{[2]} = \{1,2,3,4\}$ is a circuit. Another cycle system is given by $\mathcal{C} = \{\{1, 4, 5\}, \{1, 2, 3, 4\}\}$.
\end{example}

The motivation for cycle systems comes from the fact that they allow us to construct a collection of integer vectors whose degree sequence recovers the $h$-vector of the underlying matroid. Just as any graph $G$ gives rise to a collection of $G$-parking functions, any matroid $M$ with cycle system $\mathcal{C} = \{C_1, \dots, C_g\}$ defines a collection of nonnegative integer sequences called \emph{coparking functions}, which we denote ${\mathcal P}^*({\mathcal C})$. The basic idea is to consider sequences $(a_1, \dots, a_g)$ of nonnegative integers that cannot be `set-fired' according to the firing rules described by ${\mathcal C}$.  We refer to \cite{CycleSystems} for the precise definition and more details, but recall the main result here. In what follows, the \emph{degree} of a coparking function is the sum of its entries, and the \emph{degree vector} $(d_1, d_2, \dots)$ counts the number $d_i$ of coparking functions of degree $i$. We then have the following.

\begin{theorem}[{\cite[Proposition 4.3, Theorem 4.5]{CycleSystems}}]\label{thm:Stanley}
Suppose $M$ is a matroid with cycle system ${\mathcal C}$. Then ${\mathcal P}^*({\mathcal C})$ is a pure multicomplex whose degree vector is the $h$-vector of $M$.
\end{theorem}

A well-known conjecture of Stanley \cite{Stanley} posits that the $h$-vector of any matroid can be realized as the degree vector of some pure multicomplex. Hence \Cref{thm:Stanley} establishes Stanley's conjecture for the case of matroids that admit a cycle system. As discussed in \cite{CycleSystems}, this class includes graphic matroids of coned graphs as well as $K_{3,3}$-minor-free graphs. Also from \cite{CycleSystems}, we have the following conjecture.

\begin{conjecture}
All matroids that admit cycle systems are binary $(\ms{C}\subset \ms{B})$.
\end{conjecture}

In \Cref{sec:binary} we prove this conjecture, and then in \Cref{sec:regular} we prove that such matroids are in fact regular.

\subsection{Properties of cycle systems}

We next recall some properties of cycle systems that we use in our work.
In \cite{CycleSystems} the authors show that if $M$ is a matroid with a cycle system $\mc{C}$, then certain single-element deletions and contractions admit a cycle system that can be obtained by slightly modifying $\mc {C}$.

\begin{proposition}[{\cite[Proposition 3.6]{CycleSystems}}]\label{prop:del-con cs}
Suppose $M$ is a matroid with cycle system ${\mathcal C}$.
\begin{enumerate}
  \item Let $e\in {\mathcal C}_{[g]}\cap C_i$.  Then 
    \[
    {\mathcal C}':=  {\mathcal C} \setminus\{C_{i}\}= \{C_1,\dots,C_{i-1}, C_{i+1}, \dots, C_g\}    \]
    is a cycle system for $M\setminus e$.
  \item Let $e$ be any non-loop of $M$.  Then 
    \[
      {\mathcal C}'':=\{C_1\setminus e,...,C_g\setminus e\}
    \]
    is a cycle system for $M/e$.
\end{enumerate}
\end{proposition}

In addition, under certain conditions the property of admitting a cycle system is preserved under taking 2-sums. This establishes a large class of matroids that admit cycle systems, including the class of graphic matroids of $K_{3,3}$-minor-free graphs.

We next recall some other properties of cycle systems from \cite{CycleSystems} that will be used in our work. First, it turns out that if $M$ is a  matroid with a cycle system ${\mathcal C}$, then \emph{any} circuit of $M$ can be realized as the unique union of some subset of ${\mathcal C}$.

\begin{lemma}[{\cite[Prop. 3.7]{CycleSystems}}]\label{prop:anton}
    Suppose $M$ is a matroid with cycle system ${\mathcal C}$, and let $D$ be a circuit of $M$. Then there exists $S \subseteq [g]$ such that \({\mathcal C}_{S} = D\).
\end{lemma}

Second, while general matroids can admit cycle systems whose elements are not all circuits, this is not the case for connected matroids.

\begin{lemma}
[{\cite[Theorem 3.10]{CycleSystems}}]\label{lem:circuitsys}
Suppose $M$ is a connected matroid with cycle system ${\mathcal C}$. Then ${\mathcal C}$ is a circuit system. 
\end{lemma}

Even in a general, not-necessarily-connected matroid with a cycle system, one can always find an underlying circuit system. This fact is particularly important in our arguments because circuits are much nicer to deal with than cycles. More precisely, we have the following.

\begin{theorem}[{\cite[Theorem 3.13]{CycleSystems}}]
\label{thm:components}
    Let $W =\oplus_{i=1}^k M_i$ be a direct sum decomposition of a matroid $W$ into connected components, and let $\mathcal{E}$ be a cycle system for $W$. Up to a reordering of the components of $W$, there exist cycle systems $\mathcal{C}^{(i)}$ for $M_i$ such that $\mathcal{E}$ is the disjoint union of $\widetilde{\mathcal{C}^{i}}$, where
    \[\widetilde{\mathcal{C}^{i}} = \{C\cup K_C\mid C\in \mathcal{C}^{(i)},K_C\subseteq \oplus_{j=1}^{i-1} M_j\}.\]
    In particular, $\widetilde{\mathcal{C}^{1}} = \mathcal{C}^{(1)}$.
\end{theorem}

\begin{corollary}[{\cite[Corollary 3.14]{CycleSystems}}]\label{cor:component systems}
The following hold.
    \begin{enumerate}
 \item A matroid \(M\) admits a cycle system if and only if each of its connected components admits a cycle system.
        \item If \(M\) admits a cycle system, then it admits a circuit system.
    \end{enumerate}
\end{corollary}

\section{Base cases and reduction tools for cycle systems}\label{sec:tools}

Recall that the classes of binary and regular matroids both have excluded-minor characterizations (see \Cref{prop:binary u24} and \Cref{prop:regular minors}). Thus, to prove our main results, it suffices to show that no matroid with a cycle system can contain one of those excluded-minors. In this section, we introduce tools that allow us to extract smaller minors from a matroid while maintaining the existence of a cycle system.

Recall that \(\mathcal{C}_{[g]}\) consists of all elements that appear exactly once among all elements in a given cycle system $\mathcal{C} = \{C_1, C_2, \dots, C_g\}$. Our first tool allows us to identify when elements in \(\mathcal{C}_{[g]}\) belong to distinct cycles from ${\mathcal C}$.

\begin{lemma}\label{lem:unique circuit}
Let $M$ be a matroid with cycle system ${\mathcal C}$, and suppose $e,f \in E$ are distinct elements in \(\mathcal{C}_{[g]}\). Let  \(C_i\) be the unique cycle of \(\mathcal{C}\) containing $e$ and \(C_j\) the unique cycle containing $f$. If there exists a circuit $D\in\cir(M)$ containing  $e$ but not $f$, then $i \neq j$.
\end{lemma}

\begin{proof}
    Suppose for the sake of contradiction that some circuit \(D\) contained \(e\) but not \(f\), and yet $i = j$. Let $C = C_i = C_j$.  By \Cref{prop:anton}, \(D\) can be written as some unique union of cycles in \(\mathcal{C}\), and this unique union must include \(C\) because it is the only cycle containing \(e\). However, since \(f\in C\) uniquely as well, \(f\) must show up in any unique union containing \(C\), contradicting \(f\not\in D\). 
\end{proof}

We use this tool to establish the ``base case'' of our reduction method, showing that the relevant forbidden minors do not admit cycle systems. 

\begin{proposition}\label{prop:minors no cs}
   The matroids $U_{2,4}$, $F_7$, and $F_7^*$ do not admit cycle systems.
\end{proposition}

\begin{proof}
Recall that $U_{2,4}$, $F_7$, and $F_7^*$ are all connected matroids. Hence if any admits a cycle system, it must consist of circuits by \Cref{lem:circuitsys}.

$U_{2,4}$ is connected with corank 2, and so any cycle system must consist of two circuits. However, since every circuit of $U_{2,4}$ is size 3 and the ground set size is 4, any pair of circuits ${\mathcal C} = \{C_1, C_2\}$ must satisfy $|{\mathcal C}_{[2]}| \leq 2$, which is an independent set. This violates the unique union property, and so $U_{2,4}$ cannot have a cycle system. 

Recall that $F_7$ has $g = 4$. Suppose there exists a cycle system $\mathcal{C}$ on $F_7$. We know the total unique union $\mathcal{C}_{[4]}$ is dependent, so it contains some circuit $D$. If $|D|$ = 4, without loss of generality we may assume $D=1246$ (see \Cref{rem:Fano}). For each pair from $1,2,4,6$, since there exists a circuit of \(F_7\) that contains one element but not the other (see our labeling of circuits from \Cref{sec:prelim}), \Cref{lem:unique circuit} tells us that each of the four elements must be in a cycle that doesn't contain any of the other three. In particular, the cycle containing $1$ cannot contain $2,4,6$. However, there is no circuit containing $1$ which is contained in the set $\{1,3,5,7\}$.

If instead $|D| = 3$, without loss of generality we can assume $D = 123$. By \Cref{lem:unique circuit}, the elements $1$, $2$, and $3$ must be in separate cycles \(C_1, C_2, C_3\) of \(\mathcal{C}\), respectively. Then since these elements each appear only once in \(\mathcal{C}\), the fourth circuit \(C\) in \(\mathcal{C}\) must be disjoint from $123$, so that \(C = 4567\). Then $C_1 = 145$ or $167$, and $C_2 = 247$ or $256$. Note that in all cases, $*\{C_1,C_2,C\}$ is independent. Thus, there is no cycle system on $F_7$.

Next, suppose there exists a cycle system $\mathcal{C} = \{C_1,C_2,C_3\}$ on $F_7^*$. Then $\mathcal{C}_{[g]}$ contains a circuit, again let it be $1246$ without loss of generality. By applying \Cref{lem:unique circuit} once again, these elements must come from different circuits in the cycle system, yet this is impossible because we have 4 elements with only 3 circuits. Thus, there is no cycle system on $F_7^*$.
\end{proof}

Our strategy to prove \Cref{thm:cycle system binary} and \Cref{thm:cycle system regular} will be as follows: assume that a matroid \(M\) with a cycle system contains some minor $N$ among $U_{2,4}, F_7, F_7^*$, and then perform repeated contractions and deletions while preserving the existence of both the minor and a cycle system, leading to a contradiction according to \Cref{prop:minors no cs}. The next tool, often known as the ``scum theorem'' and originally from \cite{CrapoRota}, allows us to contract away elements outside our prescribed minor until the ranks match, after which the minor appears as a restriction.

\begin{proposition}[{\cite[Theorem 3.3.1]{Oxley}}]\label{prop:scum}
    Let $N$ be a minor of a matroid $M$. Then there is a subset $Z$ of $E(M)\setminus E(N)$ such that $M/Z$ and $N$ have the same rank, and $N$ is a restriction of $M/Z$. Moreover, if $N$ has no loops, then $Z$ can be chosen to be a flat for $M$.
\end{proposition}

We can show that this contraction preserves the existence of a cycle system on \(M/Z\):

\begin{lemma}
\label{lem:cycle system minor contraction}
 Let $M$ be a matroid with cycle system $\mathcal{C}$, and let $Z\subseteq E(M)$. Then $M/Z$ admits a cycle system.
\end{lemma}

\begin{proof}
Contracting by $Z$ is equivalent to contracting by elements of $Z$ one at a time, in any order. Contracting by any non-loop leads to a cycle system on the contraction by \Cref{prop:del-con cs}, while contracting by a loop is equivalent to deletion by a loop, which preserves having a cycle system due to \Cref{cor:component systems}. 
\end{proof}

Next, given \(M/Z\) still with cycle system and \(N\) as a minor, we wish to show that if \(M/Z\) contains more elements than \(N\), there exists an element that can be deleted while preserving the existence of a cycle system and \(N\) as a minor. This is more difficult than the contraction step because while \Cref{prop:del-con cs} guarantees the preservation of a cycle system for any non-loop contraction, it only gives such a guarantee for deletion of an element in \(\mathcal{C}_{[g]}\). However, in the cases we analyze for \(U_{2,4}, F_7\), and \(F_7^*\) minor avoidance, the matroid $M/Z$ is nearly identical to the minor $N$: by definition it has the same rank as and contains \(N\) as a restriction, and it will in most cases differ only by loops and parallel elements. The tool introduced below of \emph{unique union cancellation} will aid us in the deletion step.

For some motivation, recall that if \(\mathcal{C}\) is a cycle system for a matroid \(M\), then for any subset of cycles \(C_1,\dots, C_k\) in ${\mathcal C}$, we have that \(\mathcal{C}_{[k]}\) is dependent. We will often be in a situation where we want to include a new cycle $D$ such that the unique union of $\{C_1, \dots, C_k, D\}$ is still dependent. When we choose $D$ appropriately (namely, when $D$ intersects the $C_i$ in a large set), we will end up with a smaller dependent set.

\begin{lemma}[Unique union cancellation]
\label{lem:uuc parallel}
    Let $M$ be a connected matroid with cycle system $\mathcal{C} = \{C_1,\dots, C_g\}$, and suppose that there exist $e_1,\dots,e_k\in \mathcal{C}_{[g]}$ with $k \leq g$, $e_i\in C_i$ and in particular, all of the $e_i$ are distinct. Then for any cycle $D$ in $M\setminus \{e_1,\dots, e_k\}$, we have that
    \[\ast\{C_1, \dots, C_k, D\}\]
    is dependent.
\end{lemma}

\begin{proof}
    Let $D$ be a cycle of $M$, and write $D$ as a union of circuits \(D = D_1\cup\dots\cup D_n\). By \Cref{prop:anton}, for each \(j\in[n]\) there exists a subset \(S_j\subseteq [g]\setminus[k]\) for which \(\mathcal{C}_{S_j} = D_j\). Let \(S = \bigcup_{j\in[n]}S_j\), and consider the unique union $\mathcal{C}_{[k]\cup S}$. This set is dependent, so to prove the claim it suffices to show that 
    \[
    \mathcal{C}_{[k]\cup S}\subseteq \ast\{C_1,\dots,C_k,D\}.
    \]
    Let $e\in \mathcal{C}_{[k]\cup S}$. Then there is a unique index
    $i\in [k]\cup S$ such that $e\in C_i$. Note that \([k]\cap S = \varnothing\) by the definition of the \(S_j\)'s, so we split into cases.
    
    If $i\in [k]$, then $e\notin \bigcup_{j\in S} C_j$, hence $e\notin \mathcal{C}_{S_j} = D_j$ for any \(j\), and so \(e\not\in D\).
    Thus $e$ lies in exactly one of $C_1,\dots,C_k,D$.
    
    If $i\in S$, then since \(e\) appears uniquely in \(C_i\) and in no other \(C_j\) for \(j\neq i\in S\), it follows that $e\in \mathcal{C}_S\subseteq D$. Furthermore, since $[k]\cap S=\varnothing$, 
    $e\notin C_j$ for all $j\in [k]$. Once again, we find that $e$ lies in exactly one of $C_1,\dots,C_k,D$.
    
    In either case, $e\in \ast\{C_1,\dots,C_k,D\}$, proving the inclusion and hence the dependence.
\end{proof}

\begin{example}\label{ex:uuc}
To illustrate \Cref{lem:uuc parallel}, consider the graphic matroid $M$ on 9 elements associated to the graph depicted in \Cref{fig:matroid9}.

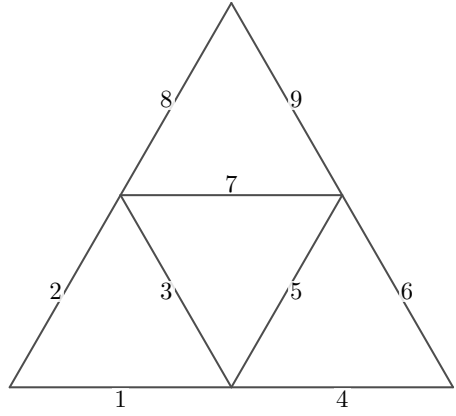
\begin{figure}[ht]
  \centering
  \scalebox{0.85}{\input{matroid9}}
  \caption{A graphic matroid on $9$ elements.}
  \label{fig:matroid9}
\end{figure}

Let $\mathcal{C} = \{C_1,C_2,C_3,C_4\}$, where $C_1 = 123, C_2 = 456, C_3 = 789, C_4 = 124689$. One can check that this defines a cycle system on $M$.

We then have $\mathcal{C}_{[g]} = \mathcal{C}_{[4]} = 357$. Choosing $e_1 = 3$, $e_2 = 5$, we get $M' = M \setminus \{3, 5\}$ with cycle system $\mathcal{C'} = \{C_3, C_4\}$. Let $D = 12467$. Then by \Cref{lem:uuc parallel} we have that $*\{C_1,C_2,D\} = 357$ is dependent. Note that $357$ is the middle triangle in the diagram.
\end{example}

We refer to \Cref{ex:uuc} for an illustration of this result. Often, we will be invoking \Cref{lem:uuc parallel} in situations where the underlying matroid has many parallel elements. To streamline its application, we include a version of the lemma which allows us to take the unique union as though all circuits were from the simplification of $M$. For the statement we will need the following definition.

\begin{definition}\label{def:circuit simplification}
   Let $M$ be a matroid and let \(\widetilde{M}\) be some fixed simplification, where an element is mapped to some representative of its parallel class. For $S \subseteq E(M)$, define \(\widetilde{S}\) to be the image of \(S\) in \(\widetilde{M}\) under this mapping.
\end{definition}

Note that if $C\in \cir(M)$ and $|C|>2$, then $\widetilde{C}\in \cir(\widetilde{M})$. The setup in the following theorem is the same as \Cref{lem:uuc parallel} except with the additional assumption that $e_1,\dots,e_k$ have no parallel elements in $M$.

\begin{theorem}\label{thm:parallel cancellation}
    Let $M$ be a connected matroid with cycle system $\mathcal{C} = \{C_1,\dots, C_g\}$, and suppose that there exist $e_1,\dots,e_k\in \mathcal{C}_{[g]}$ with $k \leq g$, $e_i\in C_i$ and in particular, all of the $e_i$ are distinct and have no parallel elements in $M$. Let $D_1,\dots, D_n$ be circuits of size $>2$ in $M\setminus \{e_1,\dots, e_k\}$, and let $D := D_1\cup \dots\cup D_n$. Then the unique union
    \[*\{\widetilde{C_1},\dots,\widetilde{C_k}, \widetilde{D}\}\]
    is dependent in $\widetilde{M}$.
\end{theorem}

\begin{proof}
    We first define a particular simplification \(\widetilde{M}\) of \(M\) by choosing representatives of the parallel classes as follows: if at least one element of a given parallel class is in \(\bigcup_{i\in[k]} C_i\), choose any one of those elements, and otherwise choose arbitrarily. We can then define \(\widetilde{D}_i\) for \(i\in[n]\) as the corresponding subset of \(\widetilde{M}\). Note that all of the \(\widetilde{D}_i\) are still circuits because we have assumed that none of the original \(D_i\) had size 2, and that \(\widetilde{D} = \bigcup_{i\in[n]} \widetilde{D}_i\).
        
    Next, for each pair of parallel elements \(e,f\) in \(\bigcup_{i\in[k]} C_i\), the circuit \(\{e,f\}\) is in \(M' := M\setminus\{e_1,\dots, e_k\}\) because we have assumed that none of the \(e_i\) have parallel elements in \(M\). Thus, letting \(P\) denote the union of all these parallel pairs, we have by \Cref{lem:uuc parallel} that 
    \[S :=\ast\{C_1,\dots, C_k, \widetilde{D}\cup P\}\]
    is dependent in \(M\). We claim that this set is a subset of \(T :=\ast\{\widetilde{C}_1,\dots, \widetilde{C}_k,\widetilde{D}\}\) in \(M\). Let \(e\in S\). Then \(e\) is in exactly one of \(C_1,\dots, C_k, \widetilde{D}\cup P\). We will show that \(e\in T\) as well. Notice that if we have $e \in P$, then we have $e \in C_i$ for some $i$ due to the definition of $P$, which would give us a contradiction.
    
    If \(e\in C_i\) for some \(i\in[k]\) then \(e\) cannot be parallel to any other element \(f\) in \(\bigcup_{i\in[k]}C_i\), or else $\{e,f\}\subseteq P$. This would mean \(e\in P\), which would give us a contradiction because \(e\) must be in exactly one of \(C_1,\dots, C_k, \widetilde{D}\cup P\). It is clear for the same reason that \(e\) itself cannot appear in any \(C_j\) for \(j\neq i\). These two observations together show that \(e\in\widetilde{C}_i\) and \(e\not\in \widetilde{C}_j\) for any other \(j\neq i\in [k]\). Combining this with the fact that $e\not\in\widetilde{D}\cup P$ and thus $e \not \in \widetilde{D}$, we get that \(e\) is in exactly one of \(\widetilde{C}_1,\dots, \widetilde{C}_k,\widetilde{D}\). We conclude that \(e\in T\).
    
    Next assume \(e\in \widetilde{D}\cup P\). From the argument above, we know that \(e\not \in P\) and hence we have \(e\in\widetilde{D}\). Notice that $e \not \in C_i$ for any \(i\in[k]\), again because \(e\) must be in exactly one of \(C_1,\dots, C_k,\widetilde{D}\cup P\). So then $e \not \in \bigcup_{i\in[k]} C_i$, which means $e \not \in \widetilde{C_i}$ for any \(i\in[k]\) because the representatives of parallel classes were chosen to prioritize the elements in $\bigcup_{i\in[k]} C_i$, one of which is guaranteed to exist for parallel classes intersecting $C_i$.
    
    We have shown that \(S\subseteq T\), which implies that \(T\) is dependent in \(M\), and hence also in $\widetilde{M}$.
\end{proof}

\section{Matroids with cycle systems are binary}
\label{sec:binary}

Recall from \Cref{prop:binary u24} that a matroid is binary if and only if it avoids $U_{2,4}$ as a minor. To prove \Cref{thm:cycle system binary} we assume that there exists a matroid with a cycle system that contains a $U_{2,4}$ minor, and argue for a contradiction. The main idea will be to use the reduction tools introduced in the previous section to contract the given matroid to a rank 2 minor that still admits a cycle system while maintaining the $U_{2,4}$ subminor. We then analyze the structure of this contracted matroid in \Cref{lem:u2n choice}, which allows us to show in \Cref{lem:delete u24} that there always exists an element we can delete while retaining a cycle system and a \(U_{2,4}\) minor.

\begin{lemma}\label{lem:u2n choice}
    Suppose $M$ is a loopless rank 2 matroid. Then a choice of $n$ elements of the ground set of $M$ forms a $U_{2,n}$ minor if and only if each of the $n$ elements is in a distinct parallel class.
\end{lemma}

\begin{proof}
    Suppose $e_1,\dots,e_n\in E(M)$. Since $M$ is loopless of rank 2, $\{e_i,e_j\}$ is a basis if and only if $e_i,e_j$ are in different parallel classes. Thus, $\{e_i,e_j\}$ is a basis for all distinct $i,j\in [n]$, i.e. \(M\mid_{\{e_1,\dots,e_n\}}\cong U_{2,n}\), if and only if each of the $n$ elements is in a distinct parallel class.
\end{proof}

\begin{lemma}\label{lem:delete u24}
    Suppose \(M\) is a rank 2 matroid with \(|E(M)| > 4\), such that $M$ admits a cycle system \(\mathcal{C}\) and contains \(U_{2,4}\) as a minor. Then there exists an element \(e\in E(M)\) such that \(M\setminus e\) also admits a cycle system and contains \(U_{2,4}\) as a minor.
\end{lemma}

\begin{proof}
If $M$ is disconnected, then some connected component of $M$ contains $U_{2,4}$ as a minor. Every other connected component of $M$ must have rank 0 and contain only loops. If $e$ is a loop, then $M\setminus e$ (which is equivalent to $M/ e$) admits a cycle system by \Cref{lem:cycle system minor contraction}, and contains $U_{2,4}$ as a minor.

If $M$ is connected, then $\mathcal{C}$ is a circuit system by \Cref{lem:circuitsys}. If there is an element $e \in \mathcal{C}_{[g]}$ not in the $U_{2,4}$ minor, we can simply delete that element and be done. Otherwise, suppose $\mathcal{C}_{[g]} \subseteq [4]$ where $[4]$ is the ground set of the $U_{2,4}$ minor. As the circuits of $U_{2,4}$ are exactly the sets of size $3$, we may assume without loss of generality that $123 \subseteq \mathcal{C}_{[g]}$.

If $M$ contains some other $U_{2,4}$ minor that does not contain $[3]$, we may delete the element that is not in that minor to satisfy the claim. Hence we may assume that all $U_{2,4}$ minors of $M$ contain $[3]$. This means that $M$ has exactly four parallel classes $\{1\},\{2\},\{3\},\{4,\dots,n\}$, for if we had more classes we would obtain another $U_{2,4}$ minor by \Cref{lem:u2n choice}. We will show that this case is not possible.

Recall that each of the parallel classes is a hyperplane and the complement of a hyperplane is a cocircuit. Then by the above, $123$ is a cocircuit. Let $C_1,C_2,C_3 \in \mathcal{C}$ be the circuits that contain $1,2,3$ respectively. From \Cref{lem:unique circuit}, we see that $C_1,C_2,C_3$ should intersect $123$ exactly once each. However, since $123$ is a cocircuit, this is not possible by \Cref{lem:intersection}.
\end{proof}

\begin{theorem}\label{thm:cycle system binary}
    Let $M$ be a matroid with cycle system $\mathcal{C}$. Then $M$ is binary.
\end{theorem}

\begin{proof}
    For a contradiction assume that \(M\) is not binary, so that \(M\) contains \(U_{2,4}\) as a minor. By \Cref{prop:scum} and \Cref{lem:cycle system minor contraction}, we can contract by some \(Z\subseteq E(M) - E(U_{2,4})\) to obtain the rank 2 matroid \(M/Z\) that admits a cycle system and contains \(U_{2,4}\) as a minor. By repeated application of \Cref{lem:delete u24}, we then get that \(U_{2,4}\) admits a cycle system, which contradicts \Cref{prop:minors no cs}. The result follows.
\end{proof}

\subsection{Corollaries for Binary Matroids}

We next collect some useful corollaries and natural questions that follow from the above results. Recall that if $M$ is a binary matroid on ground set $E$, the \emph{circuit space} of $M$ is the vector subspace of ${\mathbb F}_2^{|E|}$ spanned by the indicator vectors of the circuits in $M$. Elements in this space can be thought of as subsets of $E$, and addition coincides with symmetric difference (which we denote by $C \triangle D$). If $\{C_1, C_2, \dots, C_g\}$ is a collection of subsets of $E$ and $S \subseteq [g]$, we use $\triangle_{i\in S} C_i$ to denote the symmetric difference of the corresponding subsets.

\begin{corollary}\label{cor:circuit basis}
    Let $M$ be a connected matroid with cycle system $\mathcal{C} = \{C_1,\dots,C_g\}$. Then $\{C_1,\dots,C_g\}$ is a basis for the circuit space of $M$.
\end{corollary}

\begin{proof}
Since $M$ is connected it follows from \Cref{lem:circuitsys} that elements of ${\mathcal C}$ are in fact circuits of $M$. The result then follows from \cite[Proposition 3.15]{Components}, where it is shown that the elements of a cycle system are linearly independent in the circuit space of a connected matroid. Since $M$ is binary, this space has dimension $g$, and hence the elements of $\mathcal{C}$ form a basis.
\end{proof}

We note that the converse is not true: even if $M$ is a connected binary matroid that admits a cycle system, not all circuit bases for $M$ are necessarily circuit systems. Indeed, the appendix of \cite{CycleSystems} includes three examples of connected graphic matroids for which no fundamental circuit basis is a circuit system.

It turns out that the unique symmetric difference representation of $D$ guaranteed by \Cref{cor:circuit basis} is in fact a unique union representation. 

\begin{corollary}
\label{cor:uniquerep}
    Let $M$ be a connected matroid with circuit system $\mathcal{C} = \{C_1,\dots,C_g\}$, and let $D$ be a circuit in $M$. Then there exists a unique $S \subseteq [g]$ such that
    \[D = \triangle_{i\in S} C_i =  * \{C_i: i \in S\}.\]
\end{corollary}

\begin{proof}
    By \Cref{cor:circuit basis}, the set $\{C_1,\dots,C_g\}$ is a basis for the circuit space. Since $D$ lies in the circuit space, there exists a unique $S\subseteq [g]$ such that $D = \triangle_{i \in S} C_i$. Also, we know that:
    \[* \{C_i: i \in S\} \subseteq \triangle_{i \in S} C_i = D, \]
    and that the unique union on the left must be dependent. Since $D$ is a circuit, we conclude $* \{C_i: i \in S\} = D$.
\end{proof}


From this we see that a cycle system ${\mathcal C}$ for a connected matroid $M$ can be thought of as a special kind of basis for its circuit space: when expressing any circuit $D$ in terms of a linear combination (symmetric difference) of the basis vectors, each element of $D$ appears in \emph{precisely one} of those vectors. It remains an open question whether this property characterizes circuit systems, see \Cref{sec:Further}.


\section{Matroids with cycle systems are regular}
\label{sec:regular}

By \Cref{prop:regular minors} and the results of the previous section, to prove that any matroid \(M\) with a cycle system must be regular, we only have to show that \(M\) cannot contain either \(F_7\) or \(F_7^*\) as a minor. The proof of this will again largely rely on the reduction tools developed in \Cref{sec:tools}, and the main ideas will mirror the proof of the binary case.  Indeed, we will first contract \(M\) to a minor \(M'\) that has the same rank as the desired subminor (either \(F_7\) or \(F_7^*\)), and then seek to describe the structure of \(M'\) as we did in the binary case. We begin with some preliminary observations.


\begin{lemma}\label{lem:f7 simple}
    Let $M$ be a rank 3 matroid with cycle system $\mathcal{C}$ that contains $F_7$ as a minor. Then $\widetilde{M} \cong F_7$.
\end{lemma}

\begin{proof}
Since \(M\) admits a cycle system, it must be binary by \Cref{thm:cycle system binary}. Since in addition $M$ has rank $3$, it is representable as a matrix over \(\mathbb{F}_2\) with 3 rows. Observe that in the binary representation of \(F_7\) below \Cref{fig: Fano and Kyle}, all possible nonzero vectors appear. Hence any additional columns in our representation of \(M\) must either be zero or repeated instances of columns in \(F_7\), corresponding to loops and parallel elements. 
\end{proof}

As in the binary case, we will now use this structural result to show that there always exists a deletion that preserves both a cycle system and an \(F_7\) minor. Whenever $M$ is rank $3$ and contains $F_7$ as a minor, we will use the set $[7]$ to denote the elements of its simplification \(\widetilde{M}\).

\begin{lemma}\label{lem:delete f7}   
    Suppose $M$ is a rank 3 matroid with \(|E(M)| > 7\), such that $M$ admits a cycle system $\mathcal{C}$ and contains \(F_7\) as a minor. Then there exists an element \(e\in E(M)\) such that \(M\setminus e\) also admits a cycle system and contains \(F_7\) as a minor.
\end{lemma}

\begin{proof}
If $M$ is disconnected, then (as in the beginning of the proof of \Cref{lem:delete u24}) $M$ contains a loop $e$, and $M\setminus e$ admits a cycle system and contains an $F_7$ minor.

If $M$ is connected, then \(\mathcal{C}\) is a circuit system. Note that \(M\) has no loops since it is connected and has more than one element.  Also by \Cref{lem:f7 simple}, we have that \(\widetilde{M} \cong F_7\). We will use \([7]\) to denote the ground set of an \(F_7\) minor whose circuits are given by the same labeling as described in \Cref{sec:prelim}. If there exists an element of \(\mathcal{C}_{[g]}\) that is not in \([7]\), then we can delete that element and retain \([7]\) as an \(F_7\) minor, so we are done. Hence we assume that \(\mathcal{C}_{[g]}\subseteq [7]\). Since $\mathcal{C}_{[g]}$ is dependent it must contain a circuit, which must be either size 3 or 4 because these are the only circuit sizes on \(F_7\). We will show that neither of these cases can happen.

First suppose \(\mathcal{C}_{[g]}\) contains a circuit of size 4, which we label $1246$ without loss of generality (see \Cref{rem:Fano}). By \Cref{lem:unique circuit}, we know that $1,2,4,6$ are elements of a unique circuit in $\mathcal{C}$, which we label \(C_1, C_2, C_4, C_6\). Note that if any of $1,2,4,6$ had a parallel element in $M$, we could delete the corresponding element of $1246$ and retain an $F_7$ minor, proving the lemma. Thus, we can assume \(C_1, C_2, C_4, C_6\) are not size 2 circuits. Then their images $\widetilde{C_1},\widetilde{C_2},\widetilde{C_4},\widetilde{C_6}$ in $F_7$ are also circuits, each of which contains no other element of 1246. But $1246$ is a cocircuit of \(F_7\) and thus by \Cref{lem:intersection} cannot have size $1$ intersection with these circuits, a contradiction.

    Now suppose \(\mathcal{C}_{[g]}\) contains a circuit of size 3, which we label $123$ without loss of generality. Again let $C_1,C_2,C_3$ be the circuits from ${\mathcal C}$ that contain the elements $1,2,3$, respectively. We may again assume that $1$, $2$, and $3$ have no parallel elements in $M$, otherwise we could delete the parallel element and prove the lemma. Then according to our labeling convention for $F_7$ we have $\widetilde{C_1} = 145$ or $167$, and $\widetilde{C_2} = 247$ or $256$. We apply \Cref{thm:parallel cancellation} with the cycle $D = 4567$, and get that
    \[*\{\widetilde{C_1},\widetilde{C_2},\widetilde{D}\} = 12\quad\text{or}\quad 12x\quad\text{where}\quad x\neq 3\]
    is dependent, which is a contradiction. The result follows.
\end{proof}

\begin{lemma}\label{lem:no f7}
    Let $M$ be a matroid with cycle system $\mathcal{C}$. Then $M$ does not contain the Fano matroid $F_7$ as a minor.
\end{lemma}

\begin{proof}
    Suppose $M$ contained $F_7$ as a minor. By \Cref{prop:scum}, there exists \(Z\subseteq M\) such that $M/Z$ has rank $3$, admits a cycle system, and contains $F_7$ as a minor. By repeated applications of \Cref{lem:delete f7}, we then get that \(F_7\) admits a cycle system, which contradicts \Cref{prop:minors no cs}. 
\end{proof}

We next apply a similar argument to show that matroids with cycle systems cannot contain $F_7^*$ as a minor. It is no longer the case that the simplification of a rank 4 matroid with cycle system and \(F_7^*\) minor is unique, but we can prove something almost as good. In what follows we will refer to the matroid $AG(3,2)$ from \Cref{def:A32}.

We will also compute with contractions of matroids via matrix representations. Recall that if $M$ is a matroid with a given matrix representation, then the contraction of an element $e \in E(M)$ can be represented as follows. We first use Gaussian elimination to transform the given matrix into a row equivalent matrix where the column corresponding to $e$ is some basis vector $\vec{e}_i$. Then removing the $i$th row and column $e$ from the original matrix provides a representation for the contraction $M / e$.

\begin{remark}\label{rem:vectors}
In our arguments, we will be using various subsets of vectors in ${\mathbb F}_2^4$. After we choose a basis, there are $15$ nonzero such elements, and it will be useful to have the following labeling convention.

\[
\begin{array}{c|ccccccc|ccccccc|c}
 & 1&2&3&4&5&6&7
 & f_1&f_2&f_3&f_4&f_5&f_6&f_7 & g \\ \hline
 & 1&0&0&0&0&1&1 & 0&1&1&1&1&0&0 & 1 \\
 & 0&1&0&0&1&1&0 & 1&0&1&1&0&0&1 & 1 \\
 & 0&0&1&0&1&0&1 & 1&1&0&1&0&1&0 & 1 \\
 & 0&0&0&1&1&1&1 & 0&0&0&1&1&1&1 & 0  
\end{array}
\]

Note that the first $7$ columns provide a representation for $F_7^*$ whose dependencies agree with the labeling given in \Cref{sec:classes}. Also note that adding the last vector $g$ to this representation realizes $AG(3,2)$.

\end{remark}

\begin{proposition}\label{prop:f7* simplification}
    Let $M$ be a rank 4 matroid with cycle system $\mathcal{C}$ that contains $F_7^*$ as a minor. Then one of the following holds:
    \begin{enumerate}
        \item $\widetilde{M} \cong F_7^*$, or
        \item $M$ contains $AG(3,2)$ as a minor.
    \end{enumerate}
\end{proposition}

\begin{proof}
By \Cref{thm:cycle system binary}, $M$ is binary. Choose a representation of $M$ so that the $F_7^*$ minor is given by the first 7 columns of the matrix in \Cref{rem:vectors} above. As $M$ is rank $4$, the only other (non-parallel) elements must be chosen from the remaining columns. If column $g$ is among the vectors in $M$, then we obtain $AG(3,2)$ as a minor and the result follows. Otherwise, either \(\widetilde{M}\cong F_7^*\) or its representation contains at least one of the \(f_i\) for some  \(i\in[7]\). We will show that this latter case is not possible.

First assume $M$ contains $f_1$, and consider the submatroid of $M$ induced by the collection:
\[
\left[
\begin{array}{ccccccccc}
\scriptstyle 1&\scriptstyle 2&\scriptstyle 3&\scriptstyle 4&
\scriptstyle 5&\scriptstyle 6&\scriptstyle 7&
\scriptstyle f_1\\
\hline
1&0&0&0&0&1&1 & 0 \\
0&1&0&0&1&1&0 & 1 \\
0&0&1&0&1&0&1 & 1 \\
0&0&0&1&1&1&1 & 0
\end{array}
\right]
\]

Note that the first column is a standard basis vector, and hence deleting the first row and column provides a matrix representation of the contracted matroid $M / 1$. In this case we obtain a simple rank $3$ matroid with $7$ elements. The only such matroid up to isomorphism is $F_7$. But then $M$ admits an $F_7$ minor, contradicting \Cref{lem:no f7}. This same argument applies if we assume that the matroid $M$ contains any $f_i$ for $i\in[4]$ by considering \(M/i\). On the other hand, if \(M\) contains any \(f_i\) for \(i = 5, 6, 7\), we can apply a change of basis to make \(i\) into the first basis vector and proceed with the argument above. It is easy to verify that in all three cases, the contraction of the submatroid produces \(F_7\).
\end{proof}

\begin{lemma}\label{lem:delete f7*}
    Let $M$ be a rank 4 matroid with $|E(M)|>7$, and suppose $M$ admits a cycle system $\mathcal{C}$ and contains $F_7^*$ as a minor.
    Then there exists an element $e\in E(M)$ such that $M\setminus e$ also admits a cycle system and contains $F_7^*$ as a minor.
\end{lemma}

\begin{proof}
If $M$ is disconnected, then there exists a loop $e\in E(M)$ and $M\setminus e$ admits a cycle system and contains $F_7^*$ as a minor (as in the beginning of the proof of \Cref{lem:delete u24}).

If \(M\) is connected, then \(\mathcal{C}\) is a circuit system. First note that if \(M\) contained \(AG(3,2)\) as a minor, then since any deletion of $AG(3,2)$ yields a matroid isomorphic to $F_7^*$, any \(e\in \mathcal{C}_{[g]}\) will satisfy the claim.
    
    Hence we assume that \(M\) does not contain \(AG(3,2)\) as a minor. By \Cref{prop:f7* simplification}, $\widetilde{M}$ must be isomorphic to $F_7^*$.
    We may assume that $\mathcal{C}_{[g]}$ is contained in the set $[7]$, as otherwise we can delete the element of $\mathcal{C}_{[g]}$ outside $[7]$ and preserve the existence of a cycle system and $F_7^*$ minor. We know there is a circuit contained in $\mathcal{C}_{[g]}$, which we may assume without loss of generality is $1246$. 

 For each $i = 1,2,4,6$ let $C_i \in \cir(M)$ denote the circuit from $\mathcal{C}$ that contains $i$. As argued in \Cref{lem:delete f7}, if any of 1246 had a parallel element in \(M\), we could delete and be immediately done. Thus, we assume that $|C_i| > 2$ for $i = 1,2,4,6$. Then $\widetilde{C_1},\widetilde{C_2},\widetilde{C_4},\widetilde{C_6}$ are circuits in \(F_7^*\) that each intersect $1246$ exactly once. However, since 1246 is a cocircuit of \(F_7^*\), this contradicts \Cref{lem:intersection}.
\end{proof}

\begin{lemma}\label{lem:no f7*}
    Let \(M\) be a matroid with a cycle system \(\mathcal{C}\). Then \(M\) cannot have the Fano dual matroid \(F_7^*\) as a minor.
\end{lemma}

\begin{proof}
    Suppose $M$ contained $F_7^*$ as a minor. By \Cref{prop:scum}, there exists \(Z\subseteq M\) such that $M/Z$ is rank $4$, admits a cycle system, and contains $F_7^*$ as a minor. By repeated applications of \Cref{lem:delete f7*}, we then get that \(F_7^*\) admits a cycle system, which contradicts \Cref{prop:minors no cs}. Thus, $M$ cannot have $F_7^*$ as a minor.
\end{proof}

\begin{theorem}[Regularity of Matroids with Cycle Systems]\label{thm:cycle system regular}
    Let $M$ be a matroid with cycle system $\mathcal{C}$. Then $M$ is regular.
\end{theorem}

\begin{proof}
    From \Cref{thm:cycle system binary}, we know that $M$ is binary. Using \Cref{lem:no f7} and \Cref{lem:no f7*} we deduce that $M$ contains no $F_7$ or $F_7^*$ minor. The excluded-minor characterization in \Cref{prop:regular minors} therefore implies that \(M\) is regular.
\end{proof}

\section{Unique Union and Circuit Representation}\label{sec:uu and symdiff}

Recall from \Cref{prop:anton} that if $M$ is a matroid with cycle system $\mathcal{C}$, any circuit $D \in \cir(M)$ has a representation as a unique union $D = \mathcal{C}_T$, for some $T\subseteq [g]$.
In this section, we explore the question of whether this representation is unique. The immediate answer is no: consider the (disconnected) graphic matroid and cycle system depicted in \Cref{fig:planar_graphs}.

\begin{figure}[htbp]
  \centering
  \begin{tikzpicture}[line cap=round, line join=round, line width=0.8pt, scale=1.2]

    \draw (0,0) -- (0,1) -- (1,1) -- cycle;
    \draw (1,1) -- (1,2) -- (2,2) -- (2,1) -- cycle;
    \draw (1,1) -- (2,2);
    \node at (1, -0.4) {\Large $G$};

    \draw (3.2,0) -- (3.2,1) -- (4.2,1) -- cycle;
    \node at (3.7, -0.4) {\Large $C_1$};

    \draw (5.2,0) -- (5.2,1) -- (6.2,1) -- cycle;
    \draw (6.2,1) -- (6.2,2) -- (7.2,2) -- cycle;
    \node at (6.2, -0.4) {\Large $C_2$};

    \draw (8.2,0) -- (8.2,1) -- (9.2,1) -- cycle;
    \draw (9.2,1) -- (10.2,1) -- (10.2,2) -- cycle;
    \node at (9.2, -0.4) {\Large $C_3$};

  \end{tikzpicture}
  
  \caption{A planar graph $G$ whose corresponding (disconnected) matroid $M(G)$ has cycle system $\{C_1, C_2, C_3\}$. In this cycle system, the upper square can be represented by both $*\{C_1, C_2, C_3\}$ and $*\{C_2, C_3\}$.}
  \label{fig:planar_graphs}
\end{figure}
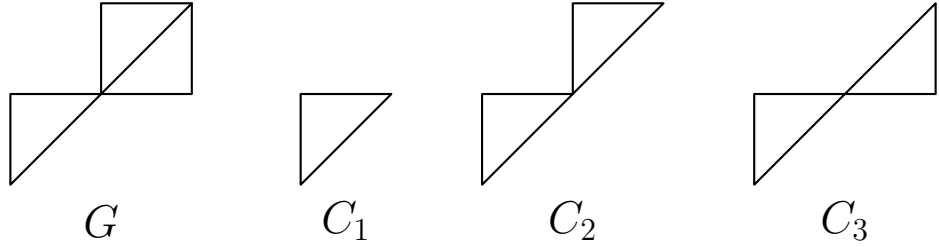

The question of whether the representation is unique when $M$ is assumed to be \emph{connected} is still open, although we conjecture that this is the case in \Cref{conj:unique}. In this section, we seek to understand the collection of subsets $S\subseteq [g]$ with $\mathcal{C}_S = D$, for some fixed circuit $D$. We then discuss some conditions on $\mathcal{C}$ that guarantee uniqueness of circuit representation using unique union. 
We first define a few terms. 

\begin{definition}[*-deletion]
    Let $M$ be a matroid with cycle system $\mathcal{C} = \{C_1,\dots,C_g\}$. If there exists $e\in \mathcal{C}_{[g]}$ with $e\in C_i$, then we call the cycle system $\mathcal{C} \setminus  C_i$ on $M\setminus e$ a $*$-deletion for $\mathcal{C}$. 
\end{definition}

The name ``$*$-deletion'' is meant to indicate that the deletion depends on the total unique union. Note that by \Cref{prop:del-con cs}, any $*$-deletion $\mathcal{C}\setminus C_i$ is a cycle system for the matroid $M\setminus e$, allowing us to iterate this process.

\begin{definition}[Reachable]
    Let $M$ be a matroid with cycle system $\mathcal{C}$. We say that a cycle system $\mathcal{C}'$ is \emph{reachable} from $\mathcal{C}$ if we may obtain $\mathcal{C}'$ from $\mathcal{C}$ via a finite sequence of $*$-deletions, so that  
    \[\mathcal{C} = \mathcal{C}^{(0)}\supset \mathcal{C}^{(1)}\supset \mathcal{C}^{(2)}\supset \mathcal{C}^{(3)}\supset \dots\supset \mathcal{C}^{(m)} = \mathcal{C}'\]
    for some \(m\geq 0\) where \(\mathcal{C}^{(i + 1)}\) is obtained from \(\mathcal{C}^{(i)}\) by \(\ast\)-deletion for each \(0\leq i < m\).
\end{definition}

Observe that if $\mc C'$ is reachable from $\mc C$ via the sequence of $*$-deletions $e_1,\dots,e_m$, then $\mc C'$ is a cycle system on $M\setminus \{e_1,\dots,e_m\}$ (and not a cycle system on $M$ itself). For convenience, we will use both notations $\mc C_S$ and $*\{C_i:i\in S\}$ interchangeably.

Given a cycle system ${\mathcal C}$ and circuit \(D\), we seek to understand the subsets \(S\subseteq [g]\) such that \(D = \mathcal{C}_S\). We use \(\mathcal{S}_D\) to denote this collection of subsets, so that
\[{\mathcal S}_D = \{S \subseteq [g]: {\mathcal C}_S = D\}.\]

We then have the following observation.


\begin{lemma}\label{lem:S_D unions}
    Let $M$ be a matroid with cycle system $\mathcal{C}$. For any circuit $D$, the set ${\mathcal S}_D$ is closed under taking unions.
\end{lemma}

\begin{proof}
    If $A,B\in \mathcal{S}_D$, then $\mathcal{C}_{A\cup B}\subseteq \mathcal{C}_A\cup \mathcal{C}_B = D$. Since $A\cup B$ is nonempty, $\mathcal{C}_{A\cup B}$ must be dependent. As $D$ is a circuit, we can conclude that $\mathcal{C}_{A\cup B} = D$.
\end{proof}

From \Cref{lem:S_D unions} we see that there exists a unique inclusion-wise maximal element in ${\mathcal S}_D$, which we will denote by $S_D^{\max}$.

\begin{lemma}\label{lem:reachable to D}
    Let $M$ be a matroid with cycle system $\mathcal{C} = \{C_1,\dots, C_g\}$, and let $D$ be a circuit of $M$. Then there exists a unique cycle system $\mathcal{C}^D$ reachable from $\mathcal{C}$ such that $D = *\{C: C \in \mathcal{C}^D\}$. Moreover, we have $\CC^D = \{C_i : i \in S_D^{\text{max}}\}$.
\end{lemma}

\begin{proof}
First note that if $D = {\mathcal C}_{[g]} = *\{C_i: i\in[g]\}$ then the result follows. Otherwise, we have that ${\mathcal C}_{[g]}$ cannot be properly contained in $D$, since ${\mathcal C}_{[g]}$ is dependent and $D$ is a circuit. Hence we can assume that there exists an element $e \in {\mathcal C}_{[g]}$ with $e \notin D$. Without loss of generality assume $e \in C_g$.

We first prove existence by iterative \(*\)-deletion. We then prove uniqueness by showing that every such deletion sequence retains all cycles indexed by \(\mathcal S_D^{\max}\).
To define $\mathcal{C}^D$, we start our $*$-deletion by deleting the element $e$, and obtaining a cycle system $\mathcal{C} \setminus C_g$ for the matroid $M \setminus e$. Note that $D$ is still a circuit of $M \setminus e$. Hence we can repeat this argument, each time reducing the number of cycles in our cycle system. Since $D$ is always a circuit of the resulting matroid, we have that the cycle system will never be empty. Hence eventually we obtain a cycle system $\mathcal{C}^D$ whose unique union is $D$.

Next we show uniqueness. To ease notation let $\CC^{\text{max}} := \{C_i : i \in S_D^{max}\}$. Suppose $\mathcal{C'}$ is a cycle system that is reachable from $\mathcal{C}$, such that the unique union of all its cycles is equal to $D$. Note that when we delete any element $e$ during this process, we must have $e \not \in D$. As before, let us write
    \[\mathcal{C} = \mathcal{C}^{(0)}\supset \mathcal{C}^{(1)}\supset\dots\supset \mathcal{C}^{(m)} = \mathcal{C}'\]
to denote the sequence of $*$-deletions we perform to obtain $\mathcal{C}'$ from $\mathcal{C}$. We claim that $\CC^{\text{max}} \subseteq \mathcal{C}^{(k)}$ for all $0 \leq k \leq m$. We proceed by induction, noting that the claim is trivial when $k=0$, as $\mathcal{C} = \mathcal{C}^{(0)}$. 

For the inductive step we assume $\CC^{\text{max}} \subseteq \CCk{k-1}$. When we go from $\CCk{k-1}$ to $\CCk{k}$, by definition we delete some element $e$ from the unique union of $\CCk{k-1}$, meaning $e$ appears in exactly one cycle of $\CCk{k-1}$. Note that any element in \(\left(\bigcup_{i\in S_D^{\max}}C_i\right)\setminus D\) has to appear in at least two of the cycles in \(\mathcal{C}^{\text{max}}\), since it does not appear in the unique union (which we assume is equal to \(D\)). However, since by induction we have $\mathcal{C}^{\text{max}} \subseteq \CCk{k-1}$, \(e\) can appear at most once in \(\mathcal{C}^{\text{max}}\), so \(e\not\in \left(\bigcup_{i\in S_D^{\max}}C_i\right)\setminus D\). Combining this with \(e\not\in D\), we get that
\(e\not\in \bigcup_{i\in S_D^{\max}} C_i\). This implies that no element of \(\mathcal{C}^{\text{max}}\) was deleted when going from \(\mathcal{C}^{(k - 1)}\) to \(\mathcal{C}^{(k)}\). Hence we get \(\mathcal{C}^{\text{max}}\subseteq \mathcal{C}^{(k)}\), finishing the induction proof of the claim.

Finally, as the unique union of all cycles in $\CC'$ is $D$ and $\CC^{\text{max}} \subseteq \CC'$, we get that $\{i : C_i \in \CC'\} = S_D^{\max}$ by the maximality of \(S_D^{\max}\). This completes the proof.
\end{proof}

We see that $S_D^{\max}$ is an inclusion-wise upper bound for the collection $\mathcal{S}_D$. The next few results introduce a set $S_D^{\min}$ as an analogous lower bound. 

\begin{proposition}\label{prop:sandwich S_D}
    Let $M$ be a matroid with cycle system $\mathcal{C} = \{C_1,\dots,C_g\}$, and let $D$ be a circuit in $M$.  Let $S_D^{\min}$ be the subset of $S_D^{\max}$ defined as:
    \[S_D^{\min} := \{i\in  S_D^{\max}\mid D\cap C_i\neq\varnothing\}.\]
    Then for all $S\in \mathcal{S}_D$,
    \[S_D^{\min}\subseteq S\subseteq S_D^{\max}.\]
\end{proposition}

\begin{proof}
    We only need to show $S_D^{\min}\subseteq S$ for $S\in \mathcal{S}_D$. Assume for the sake of contradiction that there is some $S\in\mathcal{S_D}$ such that $S_D^{\min}\not\subseteq S$. Then there exists some $i\in S_D^{\min} \setminus S$ with some $e\in C_i\cap D$. As we have $e \in  D = \mathcal{C}_{S}$, there is some $j\in S$ such that $e\in C_j$ (hence $i\neq j$).
    
    From \Cref{lem:S_D unions}, we have $S_D^{\max}\cup S\in \mathcal{S}_D$. This set contains both $i$ and $j$ because $S_D^{\min}\subset S_D^{\max}$. Both $C_i$ and $C_j$ contain $e$, which tells us that $e$ is not contained in the unique union of $\mathcal{C}_{S_D^{\max}\cup S}$, which is equal to $D$, giving us a contradiction.
\end{proof}

Note that $S_D^{\max}$ and $S_D^{\min}$ are subsets of $[g]$. Although it is true that $S_D^{\max}\in \mathcal{S}_D$, it is not necessarily true that $S_D^{\min}\in \mathcal{S}_D$. Note that if $S_D^{\min} = S_D^{\max}$, then $|\mathcal{S}_D| = 1$ and there exists a unique way to represent $D$ as a unique union. Moreover, for a circuit $D$, if every cycle in a cycle system intersects $D$ (which implies $S_D^{\min} = S_D^{\max}$), then $D$ has a unique representation using unique union.

The remainder of this section will provide another approach to proving $|\mc S_D| = 1$ using some of the same ideas. We will use the following definition in the theorem which follows.

\begin{definition}
    Let $\mc C$ be a cycle system on matroid $M$. We let $\mc C_{\ge 3}$ denote the set of elements of $E(M)$ which appear in 3 or more cycles in $\mc C$.
\end{definition}

We can now state a useful sufficient condition for $|\mc{S}_D|=1$.

\begin{proposition}\label{prop:C>=3}
    Let $M$ be a matroid with cycle system $\mathcal{C}$. Let $D\in \cir(M)$. If $\mathcal{C}^D_{\ge 3}$ is independent, then $|\mathcal{S}_D| = 1$.
\end{proposition}

\begin{proof}
    We prove the contrapositive. Suppose $|\mathcal{S}_D| > 1$ for some circuit $D$, so that $D$ has more than one representation via unique union. In particular we have some $R \in \mathcal{S}_D$ where $R \neq S_D^{\max}$.
    
    For any $i\in S_D^{\max}\setminus R$, since $S_D^{\min}\subseteq R$, we have that $C_i$ is disjoint from $D$. Now take any $e$ in the unique union $\mathcal{C}_{S_D^{\max}\setminus R}$ and let $i\in S_D^{\max}\setminus R$ be the unique index $i$ such that $e\in C_i$. If $e$ appeared once or twice among cycles of $S_D^{\max}$, then $e$ appears in either $\mathcal{C}_{S_D^{\max}}$ or $\mathcal{C}_R$, respectively. But these sets are both equal to $D$, which gives a contradiction because $D$ is disjoint from $C_i$. Recall from \Cref{lem:reachable to D} that the cycles in $\mc{C}^D$ are precisely the cycles of $\mc C$ with indices in $S_D^{\max}$. Thus, every $e$ in the unique union $\mathcal{C}_{S_D^{\max}\setminus R}$ must appear in $\mathcal{C}^D_{\geq 3}$, and thus $\mathcal{C}^D_{\geq 3}$ is dependent. 
\end{proof}

It is natural to ask whether $\mathcal{C}_{\ge 3}$ must always be independent. The answer is no in general, as can be seen from the following example.

\begin{example}\label{ex:C>=3}
    Let $M=U_{1,6}$ with ground set $[6]$. Consider the circuits $C_1 = 12, C_2 = 13, C_3 = 14, C_4 = 25,C_5 = 26$. One can check that these form a circuit system, yet we have    $\mathcal{C}_{\ge 3} = 12$, which is dependent. 
\end{example}

However, \Cref{prop:C>=3} is interested only in the cycle systems $\mc {C}^D$, which have the property that their total unique union $*\{C:C\in \mc C^D\}$ is a circuit (namely, $D$). The total unique union in \Cref{ex:C>=3} is $\{3,4,5,6\}$, which is a disjoint union of two circuits. Thus, the cycle system given by \Cref{ex:C>=3} could never appear as our reachable cycle system $\mc C^D$, so this is not a barrier to using \Cref{prop:C>=3} as a way to show uniqueness of unique union representation. This proposition is very helpful for checking that $|\mc S_D| = 1$ for specific examples, yet we would like to use it to show $|\mc S_D| = 1$ for all circuits $D$ and matroids $M$. We give the following sufficient condition to have $|\mc S_D| = 1$ over all $D$ and $M$.

\begin{theorem}\label{thm:sufficient unique}
    Suppose that for every matroid $M$ and every circuit system $\mc C$ on $M$ whose total unique union $\mc C_{[g]}$ is a circuit, we have that $\mc C_{\ge 3}$ is independent. Then $|\mc S_D| = 1$ for all matroids $M$, cycle systems $\mc C$ on $M$, and $D\in \cir(M)$.
\end{theorem}

\begin{proof}
    Observe that the hypothesis covers all possible cycle systems which could arise as $\mc C^D$. Thus, $\mc{C}^D_{\ge 3}$ is independent for all $M$, $\mc C$, and $D$, and the conclusion follows from \Cref{prop:C>=3}.
\end{proof}

In other words, \Cref{conj:C>=3 ind} implies \Cref{conj:unique}.

\section{Further questions}\label{sec:Further}

We end with some open questions and ideas for further research. Perhaps the most natural question is whether one can fully characterize $\ms{C}$. Note that $M(K_6)$ admits a cycle system, while $M(K_{3,3})$, one of its minors, does not. This implies that there cannot exist an excluded-minor classification of matroids with cycle systems, as $\ms{C}$ is not closed under deletion.

However, there is an alternative well-known classification of regular matroids due to Seymour \cite{Seymour}: a matroid is regular if and only if it can be decomposed using 1-, 2-, and 3-sums of graphic and cographic matroids, as well as an exceptional matroid $R_{10}$. $R_{10}$ can be represented over $\mathbb{F}_2$, where the ground set $E(R_{10})$ has 10 elements corresponding to the vectors in $\mathbb{F}_2^5$ with exactly 3 ones. We have computationally verified that $R_{10}\notin \ms{C}$. Thus, it is natural to ask the following question.

\begin{question}
    Is it true that no matroid $M\in \ms{C}$ can have $R_{10}$ as a minor?
\end{question}

Observe that $M(K_5)\in \ms{C}$, so the direct sum $M(K_5)\oplus M(K_5)^*\in \ms{C}$ is neither graphic nor cographic, yet it admits a cycle system. This demonstrates that $\ms{C}\not\subset \ms{G}\cup \ms{G}^*$. Recall that $M(K_{3,3})$ is graphic but not in $\ms{C}$, so the sets $\ms{C}$ and $\ms{G}\cup \ms{G}^*$ are incomparable. However, $M(K_5)\oplus M(K_5)^*$ is not connected. To find a connected matroid in $\ms{C}$ that does not lie in $\ms{G}\cup \ms{G}^*$, we appeal to the operation of 2-sum. We computationally verified that $M(K_5)\oplus_2 M(K_5)^*\in \ms{C}$. Note that $M(K_5)\oplus_2 M(K_5)^*$ is connected and is neither graphic nor cographic.

Since there cannot exist an excluded-minor characterization of $\ms{C}$, we could ask if a Seymour-like classification is possible for matroids with cycle systems.

\begin{question}
Is it possible to classify $\ms{C}$ using 1-, 2-, and 3-sums of certain matroids?
\end{question}

The proofs of \Cref{thm:cycle system binary} and \Cref{thm:cycle system regular} both employ methods of reducing a given matroid with a cycle system to a smaller such structure, in particular taking the simplification. The main technical difficulties in these proofs come from the fact that we currently do not know if the existence of a cycle system is preserved by simplification. Hence we ask the following.

\begin{question}
    If $M$ has a cycle system $\mathcal{C}$, does its simplification $\widetilde{M}$ always admit a cycle system? If so, can it be constructed from $\mathcal{C}$?
\end{question}

\subsection{Uniqueness of circuit representations}
Recall that if $M$ is a matroid with cycle system ${\mathcal C}$, then any circuit $D$ can be recovered as the unique union of some subset of ${\mathcal C}$. In \Cref{sec:uu and symdiff} we saw that this representation is not unique for an arbitrary matroid. However, if the elements of ${\mathcal C}$ consist of circuits (so that the collection ${\mathcal C}$ forms a basis for the circuit space of $M$), we make the following conjecture.

\begin{conjecture}\label{conj:unique}
Suppose $M$ is a matroid with circuit system ${\mathcal C}$, and let $D$ be a circuit of $M$. Then there exists a unique $S\subseteq [g]$ such that $D = {\mathcal C}_S$. 
\end{conjecture}

In the language of \Cref{sec:uu and symdiff}, this conjecture asserts that $|\mathcal{S}_D| = 1$ for every circuit $D$. Inspired by our observations in \Cref{sec:uu and symdiff} we also make the following conjecture.

\begin{conjecture}\label{conj:C>=3 ind}
    Let $M$ be a matroid with circuit system $\mathcal{C}$, where $\mathcal{C}_{[g]}$ is a circuit. Then $\mathcal{C}_{\ge 3}$ is independent.
\end{conjecture}

Recall that \Cref{thm:sufficient unique} demonstrates that \Cref{conj:C>=3 ind} implies \Cref{conj:unique}. We suggest attempting to prove \Cref{conj:unique} by proving \Cref{conj:C>=3 ind} instead.

Our questions of uniqueness of representations in the context of cycle systems can also be generalized as follows.
Because the circuits in a circuit system form a basis of the circuit space, the symmetric difference operator $\triangle: {\mathcal P}([g]) \rightarrow {\mathcal P}(E)$ is injective, where for a subset $S \subseteq [g]$ we define $\triangle (S)$ to be the symmetric difference of the sets $\{C_i\}_{i \in S}$.  As a strengthening of \Cref{conj:unique}, we can ask if a similar property holds for the unique union operator. 

\begin{question}
Suppose $M$ is a connected matroid with circuit system ${\mathcal C}$. Is the unique union operator $\overset{\ast}{\cup}: {\mathcal P}([g]) \rightarrow {\mathcal P}(E)$ injective? 
\end{question}

Finally, since a circuit system for a connected matroid is in particular a basis for its circuit space, we can ask to what extent the converse holds. For this, suppose $M$ is a binary matroid and ${\mathcal B} = \{B_1, B_2, \dots, B_g\}$ is a basis for its circuit space. We say that ${\mathcal B}$ is a \emph{unique union basis} if every circuit $D \in \cir(M)$ has the property that its representation $D = \triangle_{i\in S} B_i$ in terms of basis elements has the property that $D = \ast\{B_i : i\in S\}$.

\begin{question}
    Suppose $M$ is a connected binary matroid, and let ${\mathcal B} = \{B_1, B_2, \dots, B_g\}$ be a unique union basis for its circuit space. Is it true that ${\mathcal B}$ is a cycle system for $M$?
\end{question}

More generally, we are interested in sufficient conditions for a basis of the circuit space of a binary matroid to be a cycle system.

\begin{question}
    Find sufficient conditions for a basis $\mc B= \{B_1, B_2, \dots, B_g\}$ of the circuit space of a binary matroid $M$ to be a cycle system for $M$.
\end{question}

\section*{Acknowledgements}
This work was conducted as part of an REU at Texas State University in the summer of 2025, sponsored by NSF grant \#2447229. We are grateful to the NSF and Texas State for the support and the stimulating work environment. We thank Kolja Knauer and Dave Perkinson for helpful conversations related to cycle systems. Dochtermann was partially supported by Simons Foundation Grant $\#964659$. The authors used ChatGPT for editing assistance and computing examples. All mathematical content and computational results were checked by the authors, who take full responsibility for the paper.


\printbibliography

\end{document}

%% file: matroid9.tex
\begin{tikzpicture}[
  scale=1.15,
  every node/.style={font=\small},
  edge/.style={draw=black!70, line width=0.9pt, line cap=round, line join=round},
  lbl/.style={inner sep=1pt, fill=white, fill opacity=0.85, text opacity=1}
]

\coordinate (L) at (0,0);
\coordinate (R) at (6,0);
\coordinate (T) at (3,5.2);
\coordinate (B) at (3,0);
\coordinate (ML) at (1.5,2.6);
\coordinate (MR) at (4.5,2.6);

\draw[edge] (L) -- (B) node[midway, below, lbl] {1};
\draw[edge] (B) -- (R) node[midway, below, lbl] {4};

\draw[edge] (L) -- (ML) node[midway, left,  lbl] {2};
\draw[edge] (ML) -- (T)  node[midway, left,  lbl] {8};

\draw[edge] (R) -- (MR) node[midway, right, lbl] {6};
\draw[edge] (MR) -- (T) node[midway, right, lbl] {9};

\draw[edge] (ML) -- (MR) node[midway, above, lbl] {7};

\draw[edge] (ML) -- (B) node[midway, left,  lbl] {3};
\draw[edge] (MR) -- (B) node[midway, right, lbl] {5};

\end{tikzpicture}

